\documentclass[12pt]{amsart}
\usepackage{epstopdf}% To incorporate .eps illustrations using PDFLaTe\oL, etc.
\usepackage[numbers,sort&compress]{natbib}% Citation support using natbib.sty
\bibpunct[, ]{[}{]}{,}{n}{,}{,}% Citation support using natbib.sty
\makeatletter% @ becomes a letter
\def\NAT@def@citea{\def\@citea{\NAT@separator}}% Suppress spaces between citations using natbib.sty
\makeatother% @ becomes a symbol again
\usepackage{xcolor}%colors

\usepackage{amsmath}
\usepackage{amssymb}
\usepackage{amsthm}
\usepackage{amsfonts}
\usepackage{arydshln}
\usepackage{enumerate}
\usepackage[thinlines]{easybmat}
\numberwithin{equation}{section}

\DeclareMathOperator{\sees}{\text{\rm ess sup}}
\newcommand{\hsp}{\mathcal{H}}

\newcommand{\oL }{\mathcal{L}(\hsp) }

\newcommand{\h}{\mathcal{H}}

\newcommand{\ke}{{\mathbf{k}}}
\newcommand{\C}{{\mathbf{C}}}
\newcommand{\M}{{\mathbf{M}}}
\newcommand{\F}{{\mathbf{F}}}
\newcommand{\J}{{\mathbf{J}}}

\newcommand{\G}{{\mathbf{G}}}

\newcommand{\z}{{z}}

\newtheorem{theorem}{Theorem}[section]
\newtheorem{proposition}[theorem]{Proposition}

\theoremstyle{definition}

\newtheorem{remark}[theorem]{Remark}
\numberwithin{equation}{section}

\title[ MATTO's and Conjugations]{On matrix valued (asymmetric) truncated Toeplitz operators}

\author[N. G. Göğüş]{NİHAT GÖKHAN GÖĞÜŞ}
\address{Faculty of Engineering and Natural Sciences, Sabanci University Istanbul, Turkey}
\email {gokhan.gogus@sabanciuniv.edu}

\author[R. Khan]{REWAYAT KHAN}
\address{ Faculty of Engineering and Natural Sciences, Sabanci University Istanbul, Turkey}
\email{rewayat.khan@gmail.com}

\begin{document}
\begin{abstract}{ Matrix valued (asymmetric) truncated Toeplitz operators are generally not complex symmetric. In this paper, we define a new conjugation with unique properties and study its relation to matrix valued asymmetric truncated Toeplitz operators. We also explore the connections between matrix valued asymmetric truncated Toeplitz operators and Hankel operators with matrix symbols.
	}
\end{abstract}
\maketitle

\section{\bf{Introduction}}

Let $H^2$ be the classical Hardy space in the unit disk $\mathbb{D}=\{\lambda\in\mathbb{C}:|\lambda|<1\}$. Truncated Toeplitz operators (TTO's) and asymmetric truncated Toeplitz operators (ATTO's) are compressions of multiplication operator to the backward shift invariant subspaces of $H^2$ (with two possibly different underlying subspaces in the asymmetric case). Each of these subspaces is of the form $K_{\theta}=(\theta H^{2})^{\perp}=H^2\ominus \theta H^2$, where $\theta$ is a complex-valued inner function: $\theta\in H^\infty$ and $|\theta(\z)|=1$ a.e. on the unit circle $\mathbb{T}=\partial\mathbb{D}=\{\z\in\mathbb{C}: |\z|=1\}$. Since D. Sarason's paper \cite{sar} TTO's, and later on ATTO's \cite{part,part2, BCKP}, have been intensely studied (see \cite{BCT,ChFT,cg,GMR,gar3, sed} and \cite{CKP, BM,blicharz1, jl,jurasik,l,l2}).

It is natural to consider TTO's and ATTO's defined on subspaces of vector valued Hardy space $H^2(\hsp)$ with $\hsp$ a separable complex Hilbert space. %In what follows we assume that $\dim\hsp<\infty$.
 A vector valued model space $K_{\Theta}\subset H^{2}(\hsp)$ is the orthogonal complement of $\Theta H^{2}(\hsp)$, that is, $K_{\Theta}= H^{2}(\hsp)\ominus  \Theta H^{2}(\hsp)$. Here $\Theta$ is an operator-valued inner function: a function with values in $\oL $ (the algebra of all bounded linear operators on $\hsp$), analytic in $\mathbb{D}$, bounded and such that the boundary values $\Theta(\z)$ are unitary operators a.e. on $\mathbb{T}$. These spaces appear in connection with model theory of Hilbert space contractions (see \cite{NF}). Let $P_{\Theta}$ be the orthogonal projection from $L^2(\hsp)$ onto $K_{\Theta}$.

For two operator valued inner functions $\Theta_1,\Theta_2\in H^{\infty}(\oL)$ and $\Phi\in L^2(\oL )$ (again, see Sections \ref{s2} and \ref{s3} for definitions) let
\begin{equation}\label{11}
A_{\Phi}^{\Theta_{1}, \Theta_{2}}f=P_{\Theta_{2}}(\Phi f),\quad f\in K_{\Theta_{1}}\cap H^{\infty}(\hsp ).
\end{equation}

 The operator $A_{\Phi}^{\Theta_{1}, \Theta_{2}}$ is called a matrix valued asymmetric truncated Toeplitz operator (MATTO), while $A_{\Phi}^{\Theta_{1}}=A_{\Phi}^{\Theta_{1}, \Theta_{1}}$ is called a matrix valued truncated Toeplitz operator (MTTO, see \cite{KT}). Both are densely defined on $K_{\Theta_{1}}$. Let $\mathcal{MT}(\Theta_{1},\Theta_{2})$ be the set of all MATTO's of the form \eqref{11} which can be extended boundedly to the whole space $K_{\Theta_{1}}$ and for $\Theta_{1}=\Theta_{2}=\Theta$ let $\mathcal{MT}(\Theta )=\mathcal{MT}(\Theta,\Theta)$.

 Two important examples of operators from $\mathcal{MT}(\Theta )$ are the model operators
\begin{equation}\label{modl}
S_{\Theta}=A_{z }^{\Theta }=A_{zI_{\hsp}}^{\Theta }\quad\text{and}\quad S_{\Theta }^*=A_{\bar z }^{\Theta }=A_{\bar zI_{\hsp}}^{\Theta  }.
\end{equation} It is known that each $C_0$ contraction with finite defect indices is unitarily equivalent to $S_{\Theta}$ for some operator-valued inner function $\Theta$ (see \cite[Chapter IV]{NF}). On the other hand, operators from $\mathcal{MT}(\Theta_{1},\Theta_{2})$ with certain bounded analytic symbols appear as the operators intertwining $S_{\Theta_{1}}$ and $S_{\Theta_{2}}$ (see \cite[p. 238]{berc}).
 Some algebraic properties of MTTO's were studied in \cite{KT}, while the asymmetric case was investigated in \cite{RK2}.

 \medskip
The plan of the paper is the following:
 In Sections 2 and 3, we briefly give some background details of vector valued function spaces and their operators. In Section 4, we define a new conjugation corresponding to the decomposition of model space and discuss its properties. The connection between MATTO's, vectorial Hankel operators and conjugations are discussed in Section 5.

\section{\bf{Vector valued function spaces and their operators}}\label{s2}

Let $\hsp$ be a complex separable Hilbert space. In what follows, $\|\cdot\|_\hsp$ and $\langle\cdot,\cdot\rangle_{\hsp}$ will denote the norm and the inner product on $\hsp$, respectively. Moreover, we will assume that $\dim\hsp<\infty$. The space $L^{2}(\hsp)$ can be defined as
$$L^2(\h)=\{f\colon\mathbb{T}\to \hsp: f \text{ is measurable and} \int_{\mathbb{T}} \| f(\z)\|_\hsp^2\,dm(\z)<\infty  \}$$
($m$ being the normalized Lebesgue measure on $\mathbb{T}$). As usual, each $f\in L^2(\hsp)$ is interpreted as a class of functions equal to representing $f$ a.e. on $\mathbb{T}$ with respect to $m$. The space $L^2(\h)$ is a (separable) Hilbert space with the inner product given by
\begin{equation*}
\langle f,g\rangle_{L^{2}(\hsp)}=\int_{\mathbb{T}}\langle f(\z),g(\z)\rangle_{\hsp}\,dm(\z),\quad f,g\in L^2(\hsp).
\end{equation*}

Equivalently, $L^2(\hsp)$ consists of elements $f:\mathbb{T}\to \hsp$ of the form
\begin{equation}\label{1}
 \begin{array}{rl}
 f(\z)=\sum\limits_{n=-\infty}^{\infty}a_{n}\z^n&  (\text{a.e. on }\mathbb{T})\\
  \text{with}&\{a_{n}\}\subset\hsp\text{ such that}  \sum\limits_{n=-\infty}^{\infty}\|a_{n}\|_\hsp^{2}<\infty.\end{array}\end{equation}
 The $n$-th Fourier coefficient $a_n$  of $f\in L^2(\hsp)$ is determined by
\begin{equation}\label{fourier}
\langle a_n,x\rangle_\hsp=\int_{\mathbb{T}}\overline{\z}^n\langle f(\z),x\rangle_\hsp  dm(\z)\quad \text{for all }x\in\hsp.
\end{equation}
If $f\in L^2(\hsp)$ is given by \eqref{1}, then its Fourier series converges in the  $L^2(\h)$ norm and
$$\|f\|_{L^2(\hsp)}^2=\int_{\mathbb{T}} \| f(\z)\|_\hsp^2\,dm(\z)=\sum\limits_{n=-\infty}^{\infty}\|a_{n}\|_\hsp^{2}.$$
Moreover, for $\displaystyle g(\z)=\sum_{n=-\infty}^{\infty}b_{n}\z^n\in L^2(\hsp)$ we have
$$\langle f, g\rangle_{L^2(\hsp)}=\sum\limits_{n=-\infty}^\infty \langle a_n,b_n\rangle_{\hsp}.$$
For $\hsp=\mathbb{C}$ we denote $L^2=L^2(\mathbb{C})$.

The vector valued Hardy space $H^2(\hsp)$ is defined as the set of all the elements of $L^2(\hsp)$ whose Fourier coefficients with negative indices vanish, that is,
$$H^2(\hsp)=\left\{f\in L^2(\hsp): f(\z)=\sum_{n=0}^{\infty} a_n \z^n\right\}.$$
Each $f\in H^2(\hsp)$, $\displaystyle f(\z)=\sum_{n=0}^{\infty} a_n \z^n$, can also be identified with a function
$$f(\lambda)=\sum\limits_{n=0}^\infty a_n\lambda^n,\quad \lambda\in\mathbb{D},$$
analytic in the unit disk $\mathbb{D}$ (the boundary values $f(\z)$ can be obtained from the radial limits, which converge to the boundary function in the $L^2(\hsp)$ norm). Denote by $P_+$ the orthogonal projection $P_+:L^2(\hsp)\to H^2(\hsp)$,
 $$P_+\left(\sum_{n=-\infty}^{\infty} a_n \z^n\right)=\sum_{n=0}^{\infty} a_n \z^n,$$
 and let $H^2=H^2(\mathbb{C})$.

 We can also consider the spaces
 $$L^{\infty}(\hsp)=\left\{f\colon\mathbb{T}\to \hsp: \begin{array}{l} f\text{ is measurable and}\\ \ \|f\|_{\infty}=\sees_{\z\in\mathbb{T}}\|f(\z)\|_{\hsp}<\infty\end{array}  \right\}$$
 (clearly, $L^{\infty}(\hsp)\subset L^{2}(\hsp)$) and
 $$H^{\infty}(\hsp)=L^{\infty}(\hsp)\cap H^{2}(\hsp),$$
the latter seen also as the space of all bounded $\hsp$--valued functions which are analytic in $\mathbb{D}$.

Now let $\oL $ be the algebra of all bounded linear operators on $\hsp$ equipped with the operator norm $\|\cdot\|_{\oL}$. In the case $\dim\hsp=d<\infty$, each element of $\oL $ can be identified with a $d\times d$ matrix. Denote
$$L^{\infty}(\oL)=\left\{\F\colon\mathbb{T}\to \oL:\!\!\begin{array}{l} \F\text{ is measurable and}\\ \ \|\F\|_{\infty}=\sees_{\z\in\mathbb{T}}\|\F(\z)\|_{\oL}<\infty\end{array}  \right\}$$
(a function $\F\colon\mathbb{T}\to \oL$ is measurable if $\F(\cdot)x\colon\mathbb{T}\to \hsp$ is measurable for every $x\in \hsp$).
Each $\F\in L^{\infty}(\oL)$ admits a formal Fourier expansion (a.e. on $\mathbb{T}$)
	\begin{equation}\label{jeden}
\mathbf{F}(\z)=\sum_{n=-\infty}^{\infty}{F}_n \z^n\quad\text{with }\{F_n\}\subset  \oL
\end{equation}
defined by
	\begin{equation}\label{fourier2}
F_nx=\int_\mathbb{T} \overline{\z}^n \mathbf{F}(\z)x\,dm(\z) \quad\text{for }  x\in \hsp
\end{equation}
(integrated in the strong sense). Let
$$H^{\infty}(\oL)=\left\{\F\in L^2(\oL): \F(\z)=\sum_{n=0}^{\infty} F_n \z^n\right\}.$$
The space $H^\infty(\oL)$ can equivalently be defined as the space of all analytic functions $\F:\mathbb{D}\to\oL $ such that
$$\|\F\|_{\infty}=\sup_{\lambda\in\mathbb{D}}\|\F(\lambda)\|_{\oL}<\infty.$$
Each such bounded analytic $\F$ is of the form
\begin{equation}
\label{F1}\F(\lambda)=\sum_{n=0}^{\infty}{F}_n\lambda^n,\quad \lambda\in\mathbb{D},
\end{equation}
and can be identified with the boundary function
\begin{equation}\label{F2}
\F(\z)=\sum_{n=0}^{\infty}{F}_n\z^n\in L^{\infty}(\oL).
\end{equation}
Conversely, each $\F\in L^{\infty}(\oL)$ given by \eqref{F2} can be extended by \eqref{F1} to a function bounded and analytic in $\mathbb{D}$. In each case the coefficients $\{F_n\}$ can be obtained by \eqref{fourier2} and the norms $\|\cdot\|_{\infty}$ of the boundary function and its extension coincide (see \cite[p. 232]{berc}).

Note that for each $\lambda\in \mathbb{D}$ the function $\ke_\lambda(z)= (1-\bar{\lambda}z)^{-1}I_{\hsp}$ belongs to $H^{\infty}(\oL )$. Moreover, for every $x\in\hsp$ the function $\ke_\lambda x:z\mapsto \ke_\lambda(z)x$ belongs to $H^{\infty}(\hsp)$ and has the following reproducing property
$$\langle f, \ke_\lambda x\rangle_{L^2(\hsp)}=\langle f(\lambda),x\rangle_{\hsp}, \quad f\in H^2(\hsp).$$
For each $\lambda\in\mathbb{D}$ we can consider
$$\ke_{\lambda}^{\Theta}(z)=\tfrac1{1-\bar\lambda z}(I_\hsp-\Theta(z)\Theta(\lambda)^*)\in H^{\infty}(\oL).$$
For $x\in\hsp$ we will denote the function $z\mapsto \ke_{\lambda}^{\Theta}(z)x$ simply by $\ke_{\lambda}^{\Theta}x$. Then, for each $x\in\hsp$ and $\lambda\in\mathbb{D}$, the function $\ke_{\lambda}^{\Theta}x=P_{\Theta}(\ke_{\lambda}x)$  belongs to $K_{\Theta}^{\infty}=K_{\Theta}\cap H^{\infty}(\hsp )$ and has the following reproducing property
$$\langle f, \ke_{\lambda}^{\Theta}x\rangle_{L^2(\hsp)}=\langle f(\lambda), x\rangle_\hsp\quad\text{for every }f\in K_{\Theta}.$$
It follows in particular that $K_{\Theta}^{\infty}$ is a dense subset of $K_{\Theta}$.

To each $\F\in L^\infty(\oL)$ there corresponds a multiplication operator $M_\F:L^2(\hsp)\to L^2(\hsp)$: for $f\in L^2(\hsp)$,
$$(M_\F f)(\z)=\F(\z)f(\z)\quad \text{a.e. on }\mathbb{T}.$$
We will write $\F f$ instead of $M_\F f$. For a constant $\F$, that is, $\F(z)=F$ a.e. on $\mathbb{T}$ for some $F\in\oL$, we will also write $Ff$ instead of $\F f$. By $T_{\F}$ we will denote the compression of $M_{\F}$ to the Hardy space: $T_{\F}:H^2(\hsp)\to H^2(\hsp)$, $$T_{\F}f=P_+M_{\F}f\quad\text{for }f\in H^2(\hsp).$$

It is clear that $(M_\F)^*=M_{\F^*}$ and $(T_\F)^*=T_{\F^*}$, where $\F^*(\z)=\F(\z)^*$ a.e. on $\mathbb{T}$. It is also not difficult to verify that for $\F\in L^{\infty}(\oL)$ we have that $\F\in H^{\infty}(\oL)$ if and only if $M_{\F}(H^2(\hsp))\subset H^2(\hsp)$. In particular, for $M_z=M_{zI_\hsp}$ we have $M_z^*=M_{\bar{z}}=M_{\bar{z}I_\hsp}$ and $M_z(H^2(\hsp))\subset H^2(\hsp)$. The operator $S=T_z=M_{z|H^2(\hsp)}$ is called the (forward) shift operator. Its adjoint, the backward shift operator $S^*=T_{\bar z}$, is given by the formula
$$S^*f(z)=\bar z\big(f(z)-f(0)\big).$$

We can consider $\oL $ as a Hilbert space with the Hilbert--Schmidt norm and we may also define the spaces $L^2(\oL)$ and $H^2(\oL)$ as above. Recall that the norm and the corresponding inner product are defined as follows: for Hilbert-Schmidt operators $A,B\in \oL$ we have
$$\|A\|^2_2=tr(A^*A)=\sum_{e\in \varepsilon} \langle Ae, Ae\rangle_{\hsp}$$ and
$$\langle A, B\rangle=tr(B^*A)=\sum_{e\in \varepsilon} \langle Ae, Be\rangle_{\hsp}$$
$\varepsilon$ being any orthonormal basis for $\hsp$ (see \cite[Chapter 3]{GMR}). Hence for $\F, \G \in L^2(\oL)$
\begin{align*}
\langle \F, \G\rangle_{L^{2}(\oL)}
&=\int_{\mathbb{T}}\langle \F(z), \G(z)\rangle_2 dm(z)\\
&=\int_{\mathbb{T}}\sum_{e\in \varepsilon}\langle \F(z)e, \G(z)e\rangle_\hsp dm(z).
\end{align*}
Since $\langle A,B\rangle_2=\langle B^*, A^*\rangle_2$, it follows that $\langle \F,\G\rangle_{L^2(\oL)}=\langle \G^*,\F^*\rangle_{L^2(\oL)}$.\par
Since here the Hilbert--Schmidt norm and the operator norm are equivalent, we have
$$L^{\infty}(\oL )\subset L^2(\oL ),\quad H^{\infty}(\oL )\subset H^2(\oL ).$$
Moreover, it is not difficult to verify that if $\F\in L^2(\oL)$ is given by
$$\F(z)=\sum_{n=-\infty}^{\infty}F_{n}z^n, \quad F_n\in \oL,$$
where the series is convergent in the $L^2(\oL)$-norm, then
$$\F^*(z)=\F(z)^*=\sum_{n=-\infty}^{\infty}(F_{-n})^{*}z^n.$$
We thus have
$$L^2(\oL)=\left[zH^2(\oL)\right]^*\oplus H^2(\oL).$$

For $\F\in L^2(\oL )$ the operators $M_{\F}$ and $T_{\F}$ can be densely defined on $L^{2}(\hsp )$ and $H^{2}(\hsp )$, respectively.
 For more details on spaces of vector-valued functions, we refer the reader to \cite{berc,NF,RR}.

Let us now consider conjugations. A conjugation $\Gamma$ in a Hilbert space
$\hsp$ is an antilinear map $\Gamma:\hsp\longrightarrow \hsp$ such that $\Gamma^2=I_\hsp$ and
$$\langle \Gamma f,\Gamma g\rangle=\langle g, f\rangle \quad \text{for all}\quad f,g\in \hsp. $$
The importance of
conjugations comes, for example, from their connection with complex
symmetric operators. Recall that a bounded linear operator $T:\hsp\longrightarrow\hsp$ is said to be $\Gamma$-symmetric ($\Gamma$ being a conjugation on $\hsp$) if $\Gamma T\Gamma=T^*$. We say that $T$ is complex symmetric if it is $\Gamma$-symmetric with respect
to some conjugation $\Gamma$ (see, e.g., \cite{GP} for more details on conjugations
and complex symmetric operators).

In \cite{CKLP} the authors consider certain classes of conjugations in $L^2(\hsp)$. One such conjugation is $\widetilde{\J}:L^2(\hsp)\to L^2(\hsp)$ defined for a fixed conjugation $\Gamma$ in $\hsp$ by
\begin{equation}\label{J1}
(\widetilde{\J}f)(z)=\Gamma(f(z))\quad \text{a.e. on}~~\mathbb{T}.
\end{equation}
 The other one
 is $\J^*:L^2(\hsp)\to L^2(\hsp)$ defined by
\begin{equation}\label{J}
(\J^*f)(z)=\Gamma(f(\overline{z}))\quad \text{a.e. on}~~\mathbb{T}.
\end{equation}
It is not difficult to verify that for $f(z)=\sum_{n=-\infty}^{\infty}a_nz^n\in L^{2}(\hsp)$ we have
$$(\J^*f)(z)=\sum_{n=-\infty}^{\infty}\Gamma(a_n)z^n.$$
Hence, $\J^*$ is an $\M_z$-commuting conjugation, i.e, $\J^*\M_z=\M_z\J^*$, and $\J^*(H^{2}(\hsp))=H^{2}(\hsp)$, $\J^*P_+=P_+\J^*$ (see \cite[Section 4]{CKLP}). The conjugations $\widetilde{\J}$ and $\J^*$  was considered in \cite{AP} for $\text{dim}~ \hsp=1$.

For $\F\in L^\infty(\oL)$ and an arbitrary conjugation $\Gamma$ in $\hsp$ let
\begin{equation}\label{F}
\F_\Gamma(z)=\Gamma\F(z)\Gamma\quad \text{a.e on}~~~ \mathbb{T}.
\end{equation}
Then $\F_\Gamma\in L^\infty(\oL)$. As observed in \cite{CKLP}, $\F_\Gamma\in H^\infty(\oL)$ if and only if $\F^*\in H^\infty(\oL)$, and $\F_\Gamma$ is an inner function if and only if $\F$ is. Clearly, $(\F_\Gamma)_\Gamma=\F.$ Let us also observe that if $\F$ is $\Gamma$-symmetric, that is, $\Gamma\F(z)\Gamma=\F(z)^*$ a.e on $\mathbb{T}$ (or equivalently $\F(\lambda)$ is $\Gamma$-symmetric for $\lambda$ in $\mathbb{D}$, see \cite{CKLP}), then $\F_\Gamma=\widetilde{\F}$, where $\widetilde{\F}(z)=\F(\bar z)^{*}$.
 In particular
\begin{equation}\label{JM}
\J^*\M_\F=\M_\F\J^*
\end{equation}

Note that $\F_\Gamma$ is also defined for $\F\in L^{2}(\oL)$.

\section{\bf{Model spaces and their decompositions}}
In this section, we discuss the decomposition of model spaces corresponding to the product of inner functions.
Let $\Theta, \Lambda\in H^{\infty}(\oL)$ be two inner functions. Then we say that $\Lambda$ divides $\Theta$ if there exists an inner function $\Psi\in H^{\infty}(\oL)$ such that $\Theta=\Lambda\Psi$. Th scalar case was considered in \cite{CKP}.

Recall the definition of $\Gamma$-symmetry. We say that $\Phi\in L^{2}(\oL)$ is $\Gamma$-symmetric if $\Gamma\Phi(z)\Gamma=\Phi(z)^{*}$ a.e. on $\mathbb{T}$.

\medskip

If $\Lambda$ divides $\Theta$, then we have $K_{\Lambda}\subset K_{\Theta}$. Let $\Theta\Lambda^{*}=\Phi$ and $\Lambda^{*}\Theta=\Psi$, then the following result shows that if $\Theta, \Lambda$, and $\Psi$ are $\Gamma$-symmetric then $\Psi$ and $\Lambda$ commutes, and hence $\Phi=\Psi$. Moreover, the two statements in \cite[Lemma 7.2]{CKLP} can be reduced to only one.

\begin{proposition}\label{uniq}
Let $\Theta, \Lambda$, and $\Psi$  be $\Gamma$-symmetric inner functions such that $\Theta=\Lambda\Psi$. Then, $\Lambda$ and $\Psi$ commute.
\end{proposition}
\begin{proof}
Let $\Theta, \Lambda$, and $\Psi$ be $\Gamma$-symmetric inner functions such that $\Theta=\Lambda\Psi$.
Now, $\Theta$ is $\Gamma$-symmetric inner function, we have
$$\Theta^{*}=\Gamma\Theta\Gamma=\Gamma(\Lambda\Psi)\Gamma=\Gamma\Lambda\Gamma.\Gamma\Psi\Gamma=\Lambda^{*}\Psi^{*}=(\Psi\Lambda)^{*}.$$
It follows that $\Theta=\Psi\Lambda$. We also have $\Theta=\Lambda\Psi=\Psi\Lambda$.
\end{proof}

\medskip

 \begin{proposition}
Let $\Theta, \Lambda$, and $\Psi$ be inner functions  such that $\Theta=\Lambda\Psi$. Then, for any $x\in \hsp$ and $\lambda\in \mathbb{D}$, the following hold:
\begin{enumerate}[(1)]
    \item $K_{\Theta}=K_{\Lambda}\oplus \Lambda K_{\Psi}.$\label{decom}
    \item $P_{\Theta}=P_{\Lambda}+\Lambda P_{\Psi}\Lambda^{*}.$
\item $k_{\lambda}^{\Theta}(z)x=k_{\lambda}^{\Lambda}(z)x+\Lambda(z)k_{\lambda}^{\Psi}(z)\Lambda(\lambda)^{*}x.$
    \item $\widetilde{k}_{\lambda}^{\Theta}(z)x=\widetilde{k}_{\lambda}^{\Lambda}(z)\Psi(\lambda)x+\Lambda(z)\widetilde{k}_{\lambda}^{\Psi}(z)x.$
\item $P_{\Lambda}(k_{\lambda}^{\Theta}(z)x)=k_{\lambda}^{\Lambda}(z)x, \quad P_{\Lambda}(\widetilde{k}_{\lambda}^{\Theta}(z)x)=\widetilde{k}_{\lambda}^{\Lambda}(z)\Psi(\lambda)x.$
\end{enumerate}
 \end{proposition}
\begin{proof}
The proof of (1) and (2) are easy.

\noindent (3) Consider the reproducing kernel for $K_\Theta$, given by
\begin{align*}
k_{\lambda}^{\Theta}x
&=\frac{1}{1-\overline{\lambda}z}(I-\Theta(z)\Theta(\lambda)^{*})x\\
&=\frac{1}{1-\overline{\lambda}z}(I-\Lambda(z)\Psi(z)\Psi(\lambda)^{*}\Lambda(\lambda)^{*})x\\
&=\frac{1}{1-\overline{\lambda}z}(I-\Lambda(z)\Lambda(\lambda)^{*}+\Lambda(z)\Lambda(\lambda)^{*}
-\Lambda(z)\Psi(z)\Psi(\lambda)^{*}\Lambda(\lambda)^{*})x\\
&=\frac{1}{1-\overline{\lambda}z}[(I-\Lambda(z)\Lambda(\lambda)^{*})+\Lambda(z)(I
-\Psi(z)\Psi(\lambda)^{*})\Lambda(\lambda)^{*})]x\\
&=k_{\lambda}^{\Lambda}x+\Lambda(z)k_{\lambda}^{\Psi}\Lambda(\lambda)^{*}x.
\end{align*}
(4) It follows by similar arguments as in (3).

\noindent Item (5) follows from the orthogonal projection of the vectors in (3) and (4) on the model space $K_{\Lambda}$.
\end{proof}

\section{\bf{Conjugations on model spaces}}

In the scalar case, each model space $K_{\theta}$ is equipped with a natural conjugation $C_{\theta}$, defined in terms of boundary functions by
$(C_\theta f)(z)=\theta(z)\overline{z}\overline{f(z)}$. Conjugation $C_\theta$ was considered in \cite{sar}. If $\Theta\in H^\infty(\oL)$ is an inner function and $\Gamma$ is a conjugation in $\hsp$, we can similarly define $\C_{\Theta}:L^2(\hsp)\to L^2(\hsp)$ by
\begin{equation}\label{CTheta}
(\C_{\Theta}f)(z)=\Theta(z)\overline{z} \Gamma(f(z))\quad \text{a.e. on }~~\mathbb{T}.
\end{equation}
Although $\C_{\Theta}$ is obviously an antilinear isometry, it is in general not an involution. A simple computation shows that $\C_{\Theta}$ is an involution (and therefore a conjugation) if and only if $\Theta(z)\Gamma\Theta(z)\Gamma=I_{\hsp}$ a.e. on $\mathbb{T}$, i.e., if and only if $\Theta$ is $\Gamma$-symmetric. Conjugation like \eqref{CTheta} was previously considered in \cite{CKLP}.

If $\Theta$ is $\Gamma$-symmetric, then $\C_{\Theta}(\Theta H^2(\hsp))=H^2(\hsp)^{\perp}$ and so
$$\C_{\Theta}(K_\Theta)=K_\Theta.$$

Recall that in the scalar case $\hsp=\mathbb{C}$ every TTO on the model space $K_\theta$ is $C_\theta$-symmetric, i.e.,
$$C_\theta A_\varphi^\theta C_\theta=(A_\varphi^\theta)^*=A_{\overline{\varphi}}^\theta$$
(see, e.g., \cite{sar}). In that case however, the only conjugation in $\hsp$ we need to consider is $\Gamma(z)=\overline{z}$ (and each $\varphi\in L^2$ is $\Gamma$-symmetric). In the vector valued case, the equality
\begin{equation}\label{eq: 4.12}
C_{\Theta} A_{\Phi}^\Theta C_{\Theta}=A_{\Phi^*}^\Theta
\end{equation}
is not necessarily true for an arbitrary $\Phi\in L^2(\oL)$ (even though we assume here that $\Theta$ is $\Gamma$-symmetric). However, it is satisfied if also $\Phi$ is $\Gamma$-symmetric and \text{commutes} with $\Theta$ (see \cite{KT}).

Note that in that case,
$$\C_{\Theta}=\J^*\tau_\Theta,$$
where $\tau_{\Theta}: L^{2}(\hsp)\to L^{2}(\hsp)$ given by
\begin{equation}
(\tau_{\Theta}f)(z)=\bar z\Theta(\bar z)^{*}f(\bar z) \quad a.e. ~~\text{on}~~ \mathbb{T},
\end{equation}
 is a unitary operator which maps $K_{\Theta}$ onto $K_{\widetilde{\Theta}}$, where $\widetilde{\Theta}(z)=\Theta(\bar z)^{*}$.

\medskip

In this section, we study a new conjugation on model spaces. We define a conjugation that corresponds to the orthogonal decomposition of model spaces, corresponding to the inner function $\Theta=\Lambda\Psi=\Psi\Lambda$, where $\Theta, \Lambda$, and $\Psi$ are $\Gamma$-symmetric inner functions. Conjugation in Theorem \ref{conj} was considered in \cite{CKP} in case $\text{dim}~\hsp=1$.

\begin{theorem}\label{conj}
Let $\Theta, \Lambda$, and $\Psi$ be $\Gamma$-symmetric inner functions such that $\Theta=\Lambda\Psi$. Then, the operator $$C_{\Lambda,\Psi}:K_{\Theta}=K_{\Lambda}\oplus\Lambda K_{\Psi}\to K_{\Theta}=K_{\Lambda}\oplus\Lambda K_{\Psi}$$ defined by
$$C_{\Lambda,\Psi}=C_{\Lambda}\oplus\Lambda C_{\Psi}\Lambda^{*}=\begin{bmatrix}
C_{\Lambda}&0\\
0&\Lambda C_{\Psi}\Lambda^{*}
\end{bmatrix}
$$
is a conjugation on $K_{\Theta}$.
\end{theorem}
\begin{proof}
To show that $C_{\Lambda,\Psi}$ is a conjugation on $K_{\Theta}$, we have to show that $\Lambda C_{\Psi}\Lambda^*$ is a conjugation on $\Lambda K_{\Psi}$. Since $\Lambda$ and $\Psi$ are $\Gamma$-symmetric, we have $C_{\Lambda}$ and $C_{\Psi}$ as conjugations on $K_{\Lambda}$ and $K_{\Psi}$, respectively. Let $\mathcal{C}=\Lambda C_{\Psi}\Lambda^*$. First we have to show that $\mathcal{C}(\Lambda K_\Psi)=\Lambda K_\Psi$.
Let $f\in \Lambda K_\Psi$, then $f=\Lambda g$ for some $g\in K_\Psi$. Then
$$\mathcal{C}(f)=\Lambda C_\Psi\Lambda^*(\Lambda g)=\Lambda C_\Psi g \in \Lambda K_\Psi.$$ It follows that $\mathcal{C}(\Lambda K_\Psi) \subset  \Lambda K_\Psi$.
The operator $\mathcal{C}$ is involution because $$\mathcal{C}^2=\Lambda C_{\Psi}\Lambda^* \Lambda C_{\Psi}\Lambda^*=\Lambda C_{\Psi}^2\Lambda^*=I_{\Lambda K_\Psi}.$$
Hence, we have
$$\mathcal{C}(\Lambda K_\Psi)=\Lambda K_\Psi.$$
The antilinearity of $\mathcal{C}$ follows from the antilinearity of $C_\Psi$.

Now, we prove that $\mathcal{C}$ is isometric, i.e., $\|\mathcal{C}f\|=\|f\|$ for all $f=\Lambda g\in  \Lambda K_{\Psi}$, where $g\in K_{\Psi}$. Consider

\begin{align*}
\|\mathcal{C}f\|^{2}
&=\langle
\mathcal{C}f, \mathcal{C}f
\rangle\\
&=\langle
\Lambda C_\Psi \Lambda^*\Lambda g, \Lambda C_\Psi \Lambda^*\Lambda g
\rangle \\
&=\langle
 C_\Psi g, C_\Psi g
\rangle \\
&=\langle
 g, g
\rangle\\
&=\langle
 \Lambda g, \Lambda g
\rangle\\
&=\langle
 f, f
\rangle\\
&=\|f\|^{2}.
\end{align*}
Hence, $\mathcal{C}$ is a conjugation, which follows that $C_{\Lambda, \Psi}$ is a conjugation.
\end{proof}

\medskip

\begin{theorem}\label{ctheta}
Let $\Lambda,\Theta$, and $\Psi$ be $\Gamma$-symmetric inner functions such that $\Theta=\Lambda\Psi$. Then, for
any $f_{\Lambda}\in K_{\Lambda}$ and $f_{\Psi}\in K_{\Psi}$, the conjugation
$C_{\Theta}:K_{\Theta}=K_{\Lambda}\oplus\Lambda K_{\Psi}\to K_{\Theta}=K_{\Psi}\oplus \Psi K_{\Lambda}$ has the following properties:
\begin{enumerate}[(1)]
    \item $C_{\Theta}(f_{\Lambda}+\Lambda f_{\Psi})=C_{\Psi}(f_{\Psi})+\Psi C_{\Lambda}(f_{\Lambda}).$
    \item $C_{\Theta}(f_{\Psi}+\Psi f_{\Lambda})=C_{\Lambda}(f_{\Lambda})+\Lambda C_{\Psi}(f_{\Psi}).$
\end{enumerate}
\end{theorem}
\begin{proof}
(1) Let $f_{\Lambda}\in K_{\Lambda}$ and $f_{\Psi}\in K_{\Psi}$. Then,
\begin{align*}
C_{\Theta}(f_{\Lambda}+\Lambda f_{\Psi})
&=\Theta \overline{z}\Gamma(f_{\Lambda}+\Lambda f_{\Psi})\\
&=\Theta \overline{z}\Gamma(f_{\Lambda})+\Theta \overline{z}\Gamma(\Lambda f_{\Psi})\\
&=\Theta\Lambda^{*}\Lambda \overline{z}\Gamma(f_{\Lambda})+\Theta \overline{z}\Lambda^{*}\Gamma(f_{\Psi})\\
&=\Psi C_{\Lambda} (f_{\Lambda})+\Theta \Lambda^{*}\overline{z}\Gamma(f_{\Psi})\\
&=\Psi C_{\Lambda} (f_{\Lambda})+\Psi\overline{z}\Gamma(f_{\Psi})\\
&=\Psi C_{\Lambda} (f_{\Lambda})+C_{\Psi}(f_{\Psi}).
\end{align*}
(2) This part is a symmetric version of (1).
\end{proof}

\medskip

\begin{proposition}
Let $\Lambda,\Theta$, and $\Psi$ be $\Gamma$-symmetric inner functions such that $\Theta=\Lambda\Psi$. Then, the operators $C_{\Lambda,\Psi}C_{\Theta}:K_{\Theta}=K_{\Psi}\oplus\Psi K_{\Lambda}\to K_{\Theta}=K_{\Lambda}\oplus \Lambda K_{\Psi}$ and
$C_{\Theta}C_{\Lambda,\Psi}:K_{\Theta}=K_{\Lambda}\oplus\Lambda K_{\Psi}\to K_{\Theta}=K_{\Psi}\oplus \Psi K_{\Lambda}$ are unitary, and have the following properties:
\begin{enumerate}[(1)]
    \item $C_{\Lambda,\Psi}C_{\Theta}=P_{\Lambda}\Psi^{*}+\Lambda P_{\Psi},$
    \item $C_{\Theta}C_{\Lambda,\Psi}=P_{\Psi}\Lambda^{*}+\Psi P_{\Lambda}.$
\end{enumerate}
\end{proposition}
\begin{proof}
It is easy to prove that $C_{\Lambda,\Psi}C_{\Theta}$ and $C_{\Lambda,\Psi}C_{\Theta}$ are unitary operators.

To prove (1), let $f_{\Lambda}\in K_{\Lambda}$ and $f_{\Psi}\in K_{\Psi}$. Then, by Theorem \ref{ctheta}, we have
\begin{align*}
C_{\Lambda,\Psi}C_{\Theta}(f_{\Psi}+\Psi f_{\Lambda})
&=C_{\Lambda,\Psi}(C_{\Lambda}(f_{\Lambda})+\Lambda C_{\Psi} (f_{\Psi}))\\
&=C^{2}_{\Lambda}(f_{\Lambda})+\Lambda C^{2}_{\Psi} (f_{\Psi}))\\
&=f_{\Lambda}+\Lambda f_{\Psi},
\end{align*}
and
\begin{align*}
(P_{\Lambda}\Psi^{*}+\Lambda P_{\Psi})(f_{\Psi}+\Psi f_{\Lambda})
&=P_{\Lambda} (\Psi^{*}f_{\Psi})+P_{\Lambda} (f_{\Lambda}))+\Lambda P_{\Psi}(f_{\Psi})+\Lambda P_{\Psi}(\Psi f_{\Lambda})\\
&=P_{\Lambda}(f_{\Lambda})+\Lambda P_{\Psi}(f_{\Psi})\\
&=f_{\Lambda}+\Lambda f_{\Psi}.
\end{align*}
It follows (1).

(2) It follows from analogous discussion as in (1).

\end{proof}

\medskip

\begin{theorem}\cite[Theorem 4.8]{CKLP}\label{mzconj}
Let $C$ be an antilinear operator on $L^{2}(\hsp)$. Then the following are equivalent:
\begin{enumerate}[(1)]
    \item $C$ is conjugation on $L^{2}({\hsp})$ such that $M_{z}C=CM_{\bar z}$.
    \item  There is an antilinear function $C_{0}\in L^{\infty}({\oL})$ such that $C=A_{C_{0}}$ and $C_{0}(z)$ is a conjugation for a.e. on $\mathbb{T}$.
        \item For any conjugation $\Gamma$ in $\hsp$, there is $U\in L^{\infty}(\oL)$ such that $U(z)$ is a $\Gamma$-symmetric unitary operator, and $C=M_{U}\widetilde{\J}=\widetilde{\J}M_{U^{*}}$.
\end{enumerate}
\end{theorem}

 We know that $\C_{\Theta}=\J^*\tau_\Theta$, where $\J^*$ is an $M_{z}$-commuting conjugation. Hence, $\C_{\Theta}$ is an $M_{z}$-conjugation, i.e., $\C_{\Theta}M_{z}=M_{\bar z}\C_{\Theta}$. It follows from Theorem \ref{mzconj} that there is $U\in L^{\infty}(\oL)$ such that $U(z)$ is a $\Gamma$-symmetric unitary operator, and $\C_{\Theta}=M_{U}\widetilde{\J}=\widetilde{\J}M_{U^{*}}$.

\begin{theorem}
Let $\Lambda,\Theta$, and $\Psi$ be $\Gamma$-symmetric inner functions such that $\Theta=\Lambda\Psi$. Then, there are $\Gamma$-symmetric unitary functions $U, V\in L^{\infty}(\oL)$ such that $$\C_{\Lambda,\Psi}=M_{U}\widetilde{\J}\oplus M_{V}\widetilde{\J}=\widetilde{\J}M_{U^{*}}\oplus \widetilde{\J}M_{V^{*}}$$.
\end{theorem}
\begin{proof}
Denote $\M_{z}=\begin{bmatrix}
M_{z}&0\\
0&M_{z}
\end{bmatrix}
$.
Let $\Lambda, \Theta$, and $\Psi$ be $\Gamma$-symmetric inner functions such that $\Theta=\Lambda\Psi$. Then,
 $\C_{\Theta}$ and $\C_{\Psi}$ are $M_{z}$- conjugations. It implies that the conjugation

 $$\C_{\Lambda,\Psi}=C_{\Lambda}\oplus \Lambda C_{\Psi}\Lambda^{*}=\begin{bmatrix}
\C_{\Lambda}&0\\
0&\Lambda \C_{\Psi}\Lambda^{*}
\end{bmatrix}$$
 is also an $\M_{z}$-conjugation. By Theorem \ref{mzconj}, there are $\Gamma$-symmetric unitary functions $U, U^{\prime}\in L^{\infty}(\oL)$ such that $\C_{\Theta}=M_{U}\widetilde{\J}=\widetilde{\J}M_{U^{*}}$, and $\C_{\Psi}=M_{U^{\prime}}\widetilde{\J}=\widetilde{\J}M_{U^{\prime*}}$.

It follows that
\begin{align*}
\C_{\Lambda,\Psi}
&=C_{\Lambda}\oplus \Lambda C_{\Psi}\Lambda^{*}\\
&=\begin{bmatrix}
\C_{\Lambda}&0\\
0&\Lambda \C_{\Psi}\Lambda^{*}
\end{bmatrix}\\
&=\begin{bmatrix}
M_{U}\widetilde{\J}&0\\
0&\Lambda M_{U^{\prime}}\widetilde{\J}\Lambda^{*}
\end{bmatrix}\\
&=\begin{bmatrix}
M_{U}\widetilde{\J}&0\\
0&\Lambda M_{U^{\prime}}\Lambda\widetilde{\J}
\end{bmatrix}\\
&=\begin{bmatrix}
M_{U}\widetilde{\J}&0\\
0&M_{V}\widetilde{\J}
\end{bmatrix}\\
&=\begin{bmatrix}
M_{U}&0\\
0&M_{V}
\end{bmatrix}
\begin{bmatrix}
\widetilde{\J}&0\\
0&\widetilde{\J}
\end{bmatrix}\\
&=M_{U}\widetilde{\J}\oplus M_{V}\widetilde{\J}
\end{align*}
where $V=\Lambda M_{U^{\prime}}\Lambda$ is unitary function in $L^{\infty}(\oL)$. It is easy to verify that $V$ is also $\Gamma$-symmetric because $U^{\prime}$ and $\Lambda$ are $\Gamma$-symmetric. The same steps follow for the proof of other representation, i.e., $\C_{\Lambda,\Psi}=\widetilde{\J}M_{U^{*}}\oplus \widetilde{\J}M_{V^{*}}$.
\end{proof}
%\section{\bf{MATTO's and conjugations}}

\section{\bf{Connection with matrix valued Hankel operators}}

In this section, we find the connection between the conjugation, matrix valued asymmetric truncated Toeplitz operators, and matrix valued Hankel operators.

Let $P_{-}=I-P_{+}$ denote the orthogonal projection from $L^{2}(\hsp)$ on $H_{-}^{2}(\hsp)=[H^{2}(\hsp)]^{\perp}$. For $\Phi\in L^{2}(\oL)$, we define
$$H_{\Phi}: H^{2}(\hsp)\to H^{2}_{-}(\hsp),\quad \quad H_{\Phi}f=P_{-}(\Phi f);$$
for $f\in H^{2}(\hsp)$.

Similarly, for $\Psi\in L^{\infty}(\oL)$,
$$\widetilde{H}_{\Psi}:  H^{2}_{-}(\hsp)\to H^{2}(\hsp),\quad \quad \widetilde{H}_{\Psi}f=P_{+}(\Psi f).$$

\begin{remark}
Let $f\in K_{\Theta}$ and $\Phi\in L^{\infty}(\oL)$. Then
\begin{align*}
\widetilde{H}_{\Theta}H_{\Theta^{*}\Phi}f
&=\widetilde{H}_{\Theta}P_{-}(\Theta^{*}\Phi f)\\
&=P_{+}(\Theta P_{-}(\Theta^{*}\Phi f))\\
&=P_{+}(\Theta (I-P_{+})(\Theta^{*}\Phi f))\\
&=(P_{+}-\Theta P_{+}\Theta^{*})(\Phi f)\\
&=P_{\Theta}(\Phi f)\\
&=A_{\Phi}^{\Theta}f.
\end{align*}
\end{remark}

\begin{proposition}
Let $\Lambda, \Theta$, and $\Psi$ be $\Gamma$-symmetric inner functions such that $\Theta=\Lambda\Psi$. Then
\begin{enumerate}[(1)]
    \item $P_{\Theta}=\Theta P_{-}\Theta^{*}-P_{-},$
    \item $P_{\Theta}=\Theta P_{-}\Theta^{*}P_{+},$
    \item $P_{\Theta}=P_{+}\Lambda P_{-}\Lambda^{*}+\Lambda P_{\Psi} \Lambda^{*}=P_{+}\Theta P_{-}\Theta^{*},$
\end{enumerate}
\end{proposition}
\begin{proof}
By using the identities $P_{\Theta}=P_{+}-\Theta P_{+}\Theta^{*}$ and $P_{-}=I-P_{+}$, the required proof follows.
\end{proof}

\medskip

\begin{theorem}\cite[Theorem 4.5]{RYJ}\label{LT}
Let $\Theta_{1}, \Theta_{2}\in H^{\infty}(\oL)$ be two inner functions, and let $\Gamma$ be a conjugations in $\hsp$ such that $\Theta_{1}$ and $\Theta$ are $\Gamma$-symmetric. A bounded linear operator $A:K_{\Theta_{1}}\to K_{\Theta_{2}}$ belongs to $\mathcal{MT}(\Theta_{1}, \Theta_{2})$ if and only if $C_{\Theta_{2}}AC_{\Theta_{1}}$ belongs to $\mathcal{MT}(\Theta_{1}, \Theta_{2})$. More precisely, $A=A_{\Phi}^{\Theta_{1}, \Theta_{2}}\in \mathcal{MT}(\Theta_{1}, \Theta_{2})$ if and only if $C_{\Theta_{2}}AC_{\Theta_{1}}=A_{\Psi}^{\Theta_{1},\Theta_{2}}\in \mathcal{MT}(\Theta_{1}, \Theta_{2})$ with
\begin{equation}
\Psi(z)=\Gamma\Theta_{2}(z)^*\Phi(z)\Theta_{1}(z)\Gamma
=\Theta_{2}(z)\Gamma\Phi(z)\Gamma\Theta_{1}(z)^*\quad \text{a.e. on } \mathbb{T}.
\end{equation}
\end{theorem}

\medskip

 For the conjugation $C_{\Lambda,\Psi}$, we no longer have the equality in general, i.e.,  $A_{\Phi}^{\Theta,\Lambda}C_{\Lambda,\Theta}\neq C_{\Lambda, \Theta}A_{\Phi^{*}}^{\Lambda,\Theta}$. In the following theorem, we find the difference of these two operators.
\begin{theorem}
Let $\Lambda, \Theta$, and $\Psi$ be $\Gamma$-symmetric inner functions such that $\Theta=\Psi\Lambda$. Let $A_{\Phi}^{\Theta, \Lambda}\in \mathcal{T}(\Theta,\Lambda)$ and $f\in K_{\Theta}$. Then we have
$$(A_{\Phi}^{\Theta,\Lambda}C_{\Lambda,\Theta}-C_{\Lambda, \Theta}A_{\Phi^{*}}^{\Lambda,\Theta})f=(\widetilde{H}_{\Lambda}H_{\Phi}C_{\Theta}-\widetilde{H}_{\Theta}H_{\Phi}C_{\Lambda}P_{\Lambda})f.$$
\end{theorem}
\begin{proof}
First we have
\begin{align*}
A_{\Phi}^{\Lambda}P_{\Lambda}+P_{\Lambda}(\Phi\Lambda P_{\Psi}(\Lambda^{*} I_{ K_{\Theta}}))
&=A_{\Phi}^{\Lambda}P_{\Lambda}+P_{\Lambda}(\Phi\Lambda P_{\Psi}(\Lambda^{*} I_{ K_{\Theta}}))\\
&=A_{\Phi}^{\Lambda}P_{\Lambda}+P_{\Lambda}(\Phi(P_{\Theta}-P_{\Lambda}))\\
%&=A_{\Phi}^{\Lambda}P_{\Lambda}+P_{\Lambda}\Phi P_{\Theta}-P_{\Lambda}\Phi P_{\Lambda}\\
&=A_{\Phi}^{\Theta,\Lambda}.
\end{align*}
It follows that
$$A_{\Phi}^{\Lambda}C_{\Lambda,\Psi}{\mid_{K_{\Lambda}}}=A_{\Phi}^{\Lambda}C_{\Lambda}=C_{\Lambda}A_{\Phi^{*}}^{\Lambda}=C_{\Lambda,\Psi}A_{\Phi^{*}}^{\Lambda}. $$
Now, we have
\begin{align}\label{TOH}
A_{\Phi}^{\Theta,\Lambda}C_{\Lambda,\Psi}f
&=A_{\Phi}^{\Lambda}C_{\Lambda,\Psi}f+P_{\Lambda}(\Phi\Lambda P_{\Psi}(\Lambda^{*} C_{\Lambda,\Psi}f),
\end{align}
and
\begin{align*}
P_{\Lambda}(\Phi\Lambda P_{\Psi}(\Lambda^{*} C_{\Lambda,\Psi}f)
&=P_{\Lambda}(\Phi\Lambda P_{\Psi}(\Lambda^{*} C_{\Lambda,\Psi}(f_{\Psi}+\Psi f_{\Lambda}))\\
&=P_{\Lambda}(\Phi\Lambda P_{\Psi}\Lambda^{*}( C_{\Lambda}f_{\Lambda}+\Lambda C_{\Psi}f_{\Psi}))\\
&=P_{\Lambda}(\Phi\Lambda P_{\Psi}\Lambda^{*}( \Lambda \bar z \Gamma (f_{\Lambda})+\Lambda C_{\Psi}f_{\Psi}))\\
&=P_{\Lambda}(\Phi\Lambda P_{\Psi}( \Lambda^{*}\Lambda \bar z \Gamma (f_{\Lambda})+\Lambda^{*}\Lambda C_{\Psi}f_{\Psi}))\\
&=P_{\Lambda}(\Phi\Lambda P_{\Psi}(\bar z \Gamma (f_{\Lambda})+ C_{\Psi}f_{\Psi}))\\
&=P_{\Lambda}(\Phi\Lambda P_{\Psi}(\bar z \Gamma (f_{\Lambda}))+\Phi\Lambda P_{\Psi}( C_{\Psi}f_{\Psi}))\\
&=P_{\Lambda}(\Phi\Lambda P_{\Psi}( C_{\Psi}f_{\Psi}))=P_{+}\Lambda P_{-}\Lambda^{*}(\Phi\Lambda ( C_{\Psi}f_{\Psi}))\\
&=P_{+}\Lambda P_{-}(\Lambda^{*}\Lambda\Phi( C_{\Psi}f_{\Psi}))=P_{+}\Lambda P_{-}(\Phi ( C_{\Psi}f_{\Psi})).
\end{align*}
Therefore, \eqref{TOH} becomes
\begin{equation}\label{AC}
    A_{\Phi}^{\Theta,\Lambda}C_{\Lambda,\Psi}f
=C_{\Lambda,\Psi}A_{\Phi^{*}}^{\Lambda}f_{\Lambda}+P_{+}\Lambda P_{-}(\Phi ( C_{\Psi}f_{\Psi})).
\end{equation}
Similarly, we have
$$A_{\Phi}^{\Lambda,\Theta}=A_{\Phi^{*}}^{\Lambda}+\Lambda P_{\Psi}\Lambda^{*}\Phi^{*}P_{\Lambda}.$$
Therefore,
\begin{equation}\label{CA}
C_{\Lambda,\Psi}A_{\Phi^{*}}^{\Lambda,\Theta}=C_{\Lambda,\Psi}A_{\Phi^{*}}^{\Lambda}+C_{\Lambda,\Psi}\Lambda P_{\Psi}\Lambda^{*}\Phi^{*}P_{\Lambda}.
\end{equation}
We also have
$$P_{+}\Lambda P_{-}\Lambda^{*}+\Lambda P_{\Psi}\Lambda^{*}=P_{\Theta}=P_{+}\Theta P_{-}\Theta^{*}.$$
By taking the difference of \eqref{AC} and \eqref{CA}, we have
\begin{align*}
(A_{\Phi}^{\Theta,\Lambda}C_{\Lambda,\Psi}-C_{\Lambda,\Psi}A_{\Phi^{*}}^{\Lambda,\Theta})f
&=P_{+}\Lambda P_{-}(\Phi ( C_{\Psi}f_{\Psi}))-C_{\Lambda,\Psi}\Lambda P_{\Psi}(\Lambda^{*}\Phi^{*}P_{\Lambda}f)\\
&=P_{+}\Lambda P_{-}(\Phi ( C_{\Psi}f_{\Psi}))+P_{+}\Lambda P_{-}(\Phi \Psi( C_{\Lambda}f_{\Lambda}))\\
&-P_{+}\Lambda P_{-}(\Phi \Psi( C_{\Lambda}f_{\Lambda}))-\Lambda C_{\Psi}P_{\Psi}(\Lambda^{*}\Phi^{*}f_{\Lambda})\\
&=P_{+}\Lambda P_{-}(\Phi ( C_{\Psi}f_{\Psi}+\Psi C_{\Lambda}f_{\Lambda}))\\
&-P_{+}\Lambda P_{-}(\Phi \Psi( C_{\Lambda}f_{\Lambda}))-\Lambda C_{\Psi}P_{\Psi}(\Lambda^{*}\Phi^{*}f_{\Lambda})\\
&=P_{+}\Lambda P_{-}(\Phi ( C_{\Psi}f_{\Psi}+\Psi C_{\Lambda}f_{\Lambda}))\\
&-P_{+}\Lambda P_{-}(\Psi\Phi( C_{\Lambda}f_{\Lambda}))-\Lambda C_{\Psi}A_{\Lambda^{*}\Phi^{*}}^{\Lambda,\Psi}f_{\Lambda}\\
&=P_{+}\Lambda P_{-}(\Phi C_{\Theta}f)
-P_{+}\Lambda P_{-}\Lambda^{*}(\Theta\Phi ( C_{\Lambda}f_{\Lambda}))\\
&-\Lambda C_{\Psi}A_{\Lambda^{*}\Phi^{*}}^{\Lambda,\Psi}f_{\Lambda}.
\end{align*}
By using Theorem \ref{LT}, we have $C_{\Psi}A_{\Lambda^{*}\Phi^{*}}^{\Lambda,\Psi}=A_{\Omega}^{\Lambda,\Psi}C_{\Lambda}$, where
$$\Omega=\Psi\Gamma\Lambda^{*}\Phi^{*}\Lambda\Gamma=\Psi\Gamma\Phi^{*}\Gamma=\Psi\Phi=\Lambda^{*}\Theta\Phi.$$
Hence,
\begin{align*}
(A_{\Phi}^{\Theta,\Lambda}C_{\Lambda,\Psi}-C_{\Lambda,\Psi}A_{\Phi^{*}}^{\Lambda,\Theta})f
&=P_{+}\Lambda P_{-}(\Phi C_{\Theta}f)
-P_{+}\Lambda P_{-}\Lambda^{*}(\Theta\Phi ( C_{\Lambda}f_{\Lambda}))\\
&-\Lambda C_{\Psi}A_{\Lambda^{*}\Phi^{*}}^{\Lambda,\Psi}f_{\Lambda}\\
&=P_{+}\Lambda P_{-}(\Phi C_{\Theta} f)
-P_{+}\Lambda P_{-}\Lambda^{*}(\Theta\Phi ( C_{\Lambda}f_{\Lambda}))\\
&-\Lambda A_{\Omega}^{\Lambda,\Psi}C_{\Lambda} f_{\Lambda}\\
%&=P_{+}\Lambda P_{-}(\Phi C_{\Theta} f)
%-P_{+}\Lambda P_{-}\Lambda^{*}(\Theta\Phi ( C_{\Lambda}f_{\Lambda}))\\
%&-\Lambda P_{\Psi}(\Omega C_{\Lambda} f_{\Lambda})\\
&=P_{+}\Lambda P_{-}(\Phi C_{\Theta} f)
-P_{+}\Lambda P_{-}\Lambda^{*}(\Theta\Phi ( C_{\Lambda}f_{\Lambda}))\\
&-\Lambda P_{\Psi}(\Lambda^{*}\Phi\Theta C_{\Lambda} f_{\Lambda})\\
&=P_{+}\Lambda P_{-}(\Phi C_{\Theta} f)
-(P_{+}\Lambda P_{-}\Lambda^{*}+\Lambda P_{\Psi}\Lambda^{*})(\Theta\Phi ( C_{\Lambda}P_{\Lambda}f))\\
&=P_{+}\Lambda P_{-}(\Phi C_{\Theta} f)
-P_{\Theta}(\Theta\Phi ( C_{\Lambda}P_{\Lambda}f))\\
&=P_{+}\Lambda P_{-}(\Phi C_{\Theta} f)
-P_{+}\Theta P_{-}\Theta^{*}(\Theta\Phi ( C_{\Lambda}P_{\Lambda}f))\\
&=P_{+}\Lambda P_{-}(\Phi C_{\Theta} f)
-P_{+}\Theta P_{-}(\Phi ( C_{\Lambda}P_{\Lambda}f))\\
&=(\widetilde{H}_{\Lambda}H_{\Phi}C_{\Theta}-\widetilde{H}_{\Theta}H_{\Phi}C_{\Lambda}P_{\Lambda})f.
\end{align*}
\end{proof}

\noindent $\mathbf{Acknowledgment}$

\noindent Both the authors are supported by the project TUBITAK 1001, 123F356.

\end{document}